\documentclass[12pt]{amsart}
\usepackage{amsmath, amsthm, amscd, amsfonts}
\usepackage{graphicx}

\newfont{\aj}{eufm10 at12pt}
\newfont{\ajk}{eufm10 at10pt}
\theoremstyle{plain}
\newtheorem{theorem}{Theorem}[section]
\newtheorem{lemma}[theorem]{Lemma}

\newtheorem{corollary}[theorem]{Corollary}
\theoremstyle{definition}
\newtheorem{definition}[theorem]{Definition}
\newtheorem{example}[theorem]{Example}

\numberwithin{equation}{section}

\begin{document}

\title{Weak shadowing property for IFS}
\author [MEHDI FATEHI NIA ] {MEHDI FATEHI NIA }
\address{  Department of Mathematics,
Yazd University, P. O. Box 89195-741 Yazd, Iran \\ \tiny{\textrm{ e-mail:fatehiniam@yazd.ac.ir}}}

\maketitle

\begin{abstract}
In this paper we extend the notion of weak shadowing property to parameterized iterated function systems \emph{IFS} and prove some related theorems on these notion.  It is proved that every uniformly contracting  and every uniformly expanding  IFS has the weak shadowing property. Then, as an example, we give an IFS which has the  shadowing property, but fails to have the weak shadowing property. As a main result, we show that the weak shadowing property is a generic property in the set of all iterated function systems. \end{abstract}

\section{Introduction}

The notion of shadowing is an important tool for studying properties of discrete dynamical systems. From numerical point of view, if a dynamical system has the shadowing property, then numerically obtained orbits reflect the real behavior of trajectories of the systems(see \cite{[CKY],[KP],[KS]}). The so-called \emph{weak shadowing property} is introduced and studied in \cite{[CP],[MM]} and this is proved that shadowing and weak shadowing are $C^{0}$- generic properties in $H(X)$, where $H(X)$ is the set of all homeomorphisms of a compact topological space $X$. Specially, the space $X$ is one of the following:\\
(i) a topological manifold with boundary ($dim(X)\geq 2$ if $\partial X\neq\emptyset$.),
\\
(ii) a Cartesian product of a countably infinite number of manifolds with nonempty
boundary,
\\
(iii) a Cantor set,\\
then weak shadowing is a generic property in $H(X)$ \cite{[MM]}.
\\
 Iterated function systems( \textbf{IFS}), are used for the construction of deterministic fractals and have found numerous applications, in
particular to image compression and image processing \cite{[B]}. Important notions in
dynamics like attractors, minimality, transitivity, and shadowing can be extended to
IFS (see \cite{[BV],[GG]}). Gutu and Glavan defined the shadowing property for a parameterized iterated function system and prove that if a parameterized IFS is uniformly contracting, then it has the shadowing property\cite{[GG]}. \\The present paper concerns the weak shadowing property for parameterized IFS and some important result about weak shadowing property are extended to iterated function systems.
Firstly, we introduce the weak shadowing property and conjugacy on  \textbf{IFS}. In Theorem \ref{th1}  we show that if an \textbf{IFS} has the shadowing property it has the weak shadowing property. So every uniformly contracting (expanding)  \textbf{IFS} has the weak shadowing property.  In Example \ref{ea} we give an \textbf{IFS} which has  the weak shadowing property but does not have the shadowing property. Then we prove that the weak shadowing is a generic property in $\mathcal{H}_{\Lambda}(X)$.
\section{preliminaries}
Let $(X,d)$ be a complete metric space. Let us recall that a \emph{parameterized Iterated Function System(IFS)} $\mathcal{F}=\{X; f_{\lambda}|\lambda\in\Lambda\}$ is any family of continuous mappings $f_{\lambda}:X\rightarrow X,~\lambda\in \Lambda$, where $\Lambda$ is a finite nonempty set (see\cite{[GG]}).\\ Let $T=\mathbb{Z}$ or $T=\mathbb{Z}_{+}= \{n\in \mathbb{Z}:n\geq0\}$ and $\Lambda^{T}$ denote the set
of all infinite sequences $\{\lambda_{i}\}_{i\in T}$ of symbols belonging to $\Lambda$. A typical element of $\Lambda^{\mathbb{Z}_{+}}$
can be denoted as $\sigma= \{\lambda_{0},\lambda_{1},...\}$ and we use the shorted notation $$\mathcal{F}_{\sigma_{n}}=f_{\lambda_{n}} o  ...o f_{\lambda_{1}}o f_{\lambda_{0}}.$$
\begin{definition}\cite{[GG]} A sequence $\{x_{n}\}_{n\in T}$ in $X$ is called an orbit of the \textbf{IFS} $\mathcal{F}$ if there exist $\sigma\in \Lambda^{T}$ such that $x_{n+1}=f_{\lambda_{n}}(x_{n})$, for each $\lambda_{n}\in \sigma$.\\ Given $\delta>0$, a sequence $\{x_{n}\}_{n\in T}$ in $X$ is called a $\delta-$pseudo orbit of $\mathcal{F}$ if there exist $\sigma\in\Lambda^{T}$ such that for every $\lambda_{n}\in \sigma$, we have $d(x_{n+1},f_{\lambda_{n}}(x_{n}))<\delta$.\end{definition}One says that the \textbf{IFS} $ \mathcal{F}$ has the \emph{shadowing property }(on $T$) if, given $\epsilon>0$, there exists $\delta>0$ such that for any $\delta-$pseudo orbit $\{x_{n}\}_{n\in T}$ there exist an orbit $\{y_{n}\}_{n\in T}$, satisfying the inequality $d(x_{n},y_{n})\leq \epsilon$ for all $n\in T$. In this case one says that the $\{y_{n}\}_{n\in T}$ or the point $y_{0}$, $\epsilon-$ shadows the $\delta-$pseudo orbit $\{x_{n}\}_{n\in T}$.\\
\begin{definition}One says that the \textbf{IFS} $ \mathcal{F}$ has the \emph{ weak shadowing property }(on $T$) if, given $\epsilon>0$, there exists $\delta>0$ such that for any $\delta-$pseudo orbit $\texttt{x}=\{x_{n}\}_{n\in T}$ there exist an orbit $\texttt{y}=\{y_{n}\}_{n\in T}$, satisfying $\texttt{y}\subset B_{\epsilon}(\texttt{x})$.\\ Where $B_{\epsilon}(S)$ denote the set of all $x\in X$ such that $d(x,S)<\epsilon$.  \end{definition}
 Please note that if $\Lambda$ is a set whit one member then the parameterized \textbf{IFS} $ \mathcal{F}$ is an ordinary discrete dynamical system.
In this case the shadowing property for $\mathcal{F}$ is ordinary shadowing property for a discrete dynamical system.\\The parameterized \textbf{IFS}
$ \mathcal{F}=\{X; f_{\lambda}|\lambda\in\Lambda\}$ is \emph{uniformly contracting} if there exists
$$\beta= sup_{\lambda\in\Lambda} sup_{x\neq y}\frac{d(f_{\lambda}(x),f_{\lambda}(y))}{d(x,y)} $$ and this number called also the \emph{contracting ratio},
 is less than one.\\
 Respectively, we shall say that $\mathcal{F}$ is uniformly expanding if
 $$\alpha= inf_{\lambda\in\Lambda} inf_{x\neq y}\frac{d(f_{\lambda}(x),f_{\lambda}(y))}{d(x,y)}>1. $$
 We call $\alpha$ the expanding ratio \cite{[GG]}.\\
  Suppose $f,g$ are two homeomorphism on $X$ we define $d_{0}(f,g)=\max\{d(f(x),g(x)),d(f^{-1}(x),g^{-1}(x)): x\in X\}.$
 \\
Let $\mathcal{H}_{\Lambda}(X)$ denote the set of all IFS, $ \mathcal{F}=\{X; f_{\lambda}|\lambda\in\Lambda\}$ such that each $f_{\lambda}$ is a homeomorphism and for  $ \mathcal{F}=\{X; f_{\lambda}|\lambda\in\Lambda\}$,$ \mathcal{G}=\{X; g_{\lambda}|\lambda\in\Lambda\}\in\mathcal{H}_{\Lambda}(X)$ Let
$$\rho( \mathcal{F},\mathcal{G})=\max_{\lambda,\mu \in\Lambda}\{d(f_{\lambda}(x),g_{\mu}(x)),d(f^{-1}_{\lambda}(x),g^{-1}_{\mu}(x)): x\in X\}.$$ Clearly
  $\rho$ is a complete metric on $\mathcal{H}_{\Lambda}(X)$. \\
We recall that the space $X$ is homogeneous if for $\epsilon>0$ we can find $\delta>0$ which if $\{x_{1},...,x_{n}\},\{y_{1},...,y_{n}\}\subset X$ are two sets of disjoint points satisfying $d(x_{i},y_{i})\leq \delta$,for all $1\leq i\leq n$, then there exist a homeomorphism $h:X\rightarrow X$ with $d_{0}(h,id_{X})\leq \epsilon$ and $h(x_{i})=y_{i},~1\leq i\leq n $\cite{[MM]}.
\begin{definition}
Suppose $(X,d)$ and $(Y,d^{'})$ are compact metric spaces and $\Lambda$ is a finite set. Let $ \mathcal{F}=\{X; f_{\lambda}|\lambda\in\Lambda\}$ and $ \mathcal{G}=\{Y; g_{\lambda}|\lambda\in\Lambda\}$ are two $IFS$ which $f_{\lambda}:X\rightarrow X$ and  $g_{\lambda}:Y\rightarrow Y$ are continuous maps for all $\lambda\in\Lambda$. We say that $ \mathcal{F}$ is said to be topologically conjugate to $ \mathcal{G}$  if there is a homeomorphism  $h:X\rightarrow Y$  such that $g_{\lambda}=h o f_{\lambda} o h^{-1}$, for all $\lambda\in\Lambda$. In this case, $h$ is called a topological conjugacy. \end{definition}
\section{Weak shadowing property for iterated function systems}
In this section we investigate the structure of parameterized \emph{IFS} whit the weak shadowing property.
It is well known that if $f:X\rightarrow X$ and $g:Y\rightarrow Y$ are conjugated then $f$ has the weak shadowing property if and only if so does $g$. In the next theorem we extend this property for \emph{iterated function systems }.
\begin{theorem}
Suppose $(X,d)$ and $(Y,d^{'})$ are compact metric spaces and $\Lambda$ is a finite set. Let $ \mathcal{F}=\{X; f_{\lambda}|\lambda\in\Lambda\}$ and $ \mathcal{G}=\{Y; g_{\lambda}|i\in\Lambda\}$ are two $IFS$ which $f_{\lambda}:X\rightarrow X$ and  $g_{\lambda}:X\rightarrow X$ are continuous maps for all $\lambda\in\Lambda$. Suppose that $ \mathcal{F}$ is topologically conjugate to $ \mathcal{G}$, then $\mathcal{F}$ has the weak shadowing property if and only if so does $\mathcal{G}$.
\end{theorem}
\begin{proof}
Suppose that $ \mathcal{F}$ has the weak shadowing property, we prove that $\mathcal{G}$ also have this property. Fix $\epsilon>0$ and consider $h:X\rightarrow Y$ as the conjugacy map between $\mathcal{F}$ and $ \mathcal{G}$. Since $h$ is a homeomorphism then there exists $\epsilon_{1}>0$ such that $d(a,b)<\epsilon_{1}$, implies $d^{'}(h(a),h(b))<\epsilon$. Let $\delta_{1}>0$ be an $\epsilon_{1}$ modulus of weak shadowing for $ \mathcal{F}$, there is $ \delta>0$ such that $d^{'}(x,y)<\delta$ implies that $d(h^{-1}(x),h^{-1}(y))<\delta_{1}.$\\
Now, Suppose that $\mathbf{x}=\{x_{i}\}_{i\geq 0}$ is a $\delta-$pseudo orbit for $\mathcal{G}$. Then $\mathbf{x}^{'}=\{h^{-1}(x_{i})\}_{i\geq 0}$ is a $\delta_{1}-$pseudo orbit for $\mathcal{F}$. Since $\mathcal{F}$ has the weak shadowing property then there exist an orbit  $\mathbf{y^{'}}=\{y_{i}\}_{i\geq 0}$ in $\mathcal{F}$ such that $\mathbf{y}^{'}\subset B_{\epsilon_{1}}(\mathbf{x}^{'})$. So, $\mathbf{y}\subset B_{\epsilon}(\mathbf{x})$, where $\mathbf{y}=\{h(y_{i})\}_{i\geq 0}$ is an orbit of $ \mathcal{G}$.
\end{proof}
By shadowing and weak shadowing definitions for \textbf{IFS} we have the following theorem.
\begin{theorem}\label{th1}
Let $X$ be a complete  metric space, if $ \mathcal{F}=\{X; f_{\lambda}|\lambda\in\Lambda\}$ has the shadowing property then it has the weak shadowing property.
\end{theorem}
So, Theorem \ref{th1}, Theorems 2.1 and 2.2 in \cite{[GG]} we have the following results.
\begin{corollary}\label{co1}
 If a parameterized IFS $\mathcal{F} = \{X ; f_{\lambda} | \lambda\in\Lambda\}$ is uniformly
contracting, then it has the weak shadowing property on $\mathbb{Z}_{+}$.
 \end{corollary}
 \begin{corollary}\label{co2}
 If a parameterized IFS $\mathcal{F} = \{X ; f_{\lambda} | \lambda\in\Lambda\}$ is uniformly
expanding and if each function $f_{\lambda} (\lambda\in\Lambda)$ is surjective, then the IFS has the
weak shadowing property on $\mathbb{Z}_{+}$.
\end{corollary}
The following example shows that the inverse of Theorem\label{th1} is not true.
\begin{example}\label{ea}
Consider the unit circle $\mathbb{S}^{1}$ with coordinate $x\in[0,1)$. Suppose that $0<\beta_{1},\beta_{2}<1$ are two distinct irrational numbers and $f_{i}$ is homeomorphisms on $S^{1}$ defined by $f_{i}(x)=x+\beta_{i}$, for $i\in\{0,1\}$. Let $\mathcal{F} = \{\mathbb{S}^{1} ; f_{1}, f_{2} \}$. Since every orbit of $f_{1}$ is an orbit of $\mathcal{F}$ that is dense in $\mathbb{S}^{1}$, then $\mathcal{F}$ has the weak shadowing property\cite{[SP]}.\\
Now, suppose that $\beta_{2}-\beta_{1}=\frac{1}{2}$. We show that $\mathcal{F}$ does not have the shadowing property.\\
To obtain a contradiction, we assume that $\mathcal{F}$ has the shadowing property. Take $\epsilon=\frac{1}{5}$ and $\delta>0$ be the corresponding number for shadowing property. Let $\alpha$ be a rational number which $| \alpha- \beta_{1}|<\delta$ and  $g: S^{1}\rightarrow S^{1}$ be a homeomorphism defined by $g(x)=x+\alpha$. This is clear that every orbits of $g$ is a $\delta-$pseudo orbit of $f_{1}(x)=x+\beta$. Since $\alpha$ is a rational number, then there is $m\in \mathbb{N}$ such that $g^{m}$ is identity map. Let $\sigma= \{...,\lambda_{-1},\lambda_{0},\lambda_{1},...\}$ be an arbitrary sequence in $\{1,2\}^{\mathbb{Z}}$.\\
\textbf{Claim:} \emph{For every $p\in\mathbb{S}^{1}$, the sets $\{\mathcal{F}_{\sigma_{mk}}(p)\}_{k\geq 0}$ is dense in $\mathbb{S}^{1}$ }.\\
So, for any $x, p \in \mathbb{S}^{1}$, $\{x, g(x),...,g^{m-1}(x),...\}$ is a $\delta$-pseudo orbit of $\mathcal{F}$, but there exists $k>0$ that $d(g^{km}(x),\mathcal{F}_{\sigma_{mk}}(p))>\frac{1}{5}$.\\
Hence $\mathcal{F}$ does not have the shadowing property.
\end{example}
\textbf{Proof of Claim}. For any $n>0$, let $A_{n}=\{\lambda_{i}; \lambda_{i}=2,  1\leq i\leq n \}$ and $n_{2}$  be the cardinality of the set $ A_{n}$. Then
  $\mathcal{F}_{\sigma_{mk}}(p)=
 p+mk\beta_{1} +\frac{mk_{2}}{2}  (mod~ 1)$.
 By a similar argument to that given in ( \cite{[PM]} Example 2), we can show that   $\{\mathcal{F}_{\sigma_{mk}}(p)\}_{k\geq 0}$ is dense in $\mathbb{S}^{1}$ .
\begin{lemma}\label{le1}
Suppose $\mathbf{x}=\{x_{i}\}_{i=0}^{n}$ is a $\delta-$pseudo orbit. there exist a $2\delta-$pseudo trajectory $\mathbf{y}=\{y_{i}\}_{i=0}^{n}$ such that $\mathbf{x}\subset B_{\epsilon}(\mathbf{y})$ and $y_{i}\neq y_{j}$ for all $i\neq j$.
\end{lemma}
\begin{proof}
Consider a finite sequence $\{\lambda_{i}\}_{i=0}^{n-1}\subset\Lambda$  such that $d(f_{\lambda_{i}}(x_{i},x_{i+1})<\delta$, for all $0\leq i \leq n-1$. Since $X$ is a compact space and has no isolated points then we can find a sequence $\mathbf{y}=\{y_{i}\}_{i=0}^{n}$ of distinct points such that $d(f_{\lambda}(x_{i}),y_{i})<\frac{\delta}{2}$ and $d(x_{i+1},y_{i+1})<\frac{\delta}{2}$, for all $\lambda \in \Lambda$ and all $0\leq i\leq n-1$. So $d(f_{\lambda_{i}}(y_{i}),y_{i+1})\leq d(f_{\lambda_{i}}(y_{i}),f_{\lambda_{i}}(x_{i}))+d(f_{\lambda_{i}}(x_{i}),x_{i+1})+d(x_{i+1},y_{i+1})<\frac{\delta}{2}+ \frac{\delta}{2}+\delta=2\delta$, for $0\leq i\leq n-1$.
\end{proof}
\begin{lemma}\label{le2}
Let $\mathrm{x}=\{x_{i}\}_{i\in \mathbb{Z}}$ be a sequence in $X$ and $\mathrm{x}_{n}=  \{x_{i}\}_{i=-n}^{n}$, for every $n\geq 1$. For each $\epsilon>0$ there exist $k\in\mathbb{N}$ such that $\mathrm{x}\subset B_{\epsilon}(\mathrm{x}_{k})$.
\end{lemma}
\begin{proof}
This is clear that $\{X-\overline{\mathrm{x}}, B_{\epsilon}(\mathrm{x}_{n}):~n\geq 1 \}$ where $\overline{\mathrm{x}}$ is closure of $\mathrm{x}$, is an open cover for $X$ and has a finite subcover. Suppose $X\subset (X-\overline{\mathrm{x}})\cup  B_{\epsilon}(\mathrm{x}_{n_{1}})\cup...\cup B_{\epsilon}(\mathrm{x}_{n_{l}})$, for some $n_{1}< n_{2}<...<n_{l}$. Since $\mathrm{x}\cap (X-\overline{\mathrm{x}})=\emptyset$ and $B_{\epsilon}(\mathrm{x}_{n_{1}})\subset B_{\epsilon}(\mathrm{x}_{n_{2}})\subset...\subset B_{\epsilon}(\mathrm{x}_{n_{l}})$, then $\mathrm{x}\subset B_{\epsilon}(\mathrm{x}_{n_{l}})$.
\end{proof}
Next theorem is the main result of this paper and the main idea of proof is the same as that of \cite{[MM]}.
\begin{theorem}\label{th1}
If the space $X$is a generalized homogeneous then the $\mathcal{F}-$weak shadowing property is generic in $\mathcal{H}_{\Lambda}(X)$.
\end{theorem}
\begin{proof}
Given $\epsilon>0$ and $V=\{V_{1},...,V_{k}\}$ be a cover of $X$ consisting of open sets with diameters less than $\epsilon$. Suppose  $\mathcal{F}\in \mathcal{H}_{\Lambda}(X)$ and take $J_{\mathcal{F}}$ is the family of sets $L\subset\{1,2,...,k\}$ such that there exist an obit $\mathrm{x}=\{x_{i}\}_{i\in \mathbb{Z}}$ of $\mathcal{F}$ satisfying $\mathrm{x}\cap V_{j}\neq\emptyset$, for all $j\in L$.\\
\textbf{Claim.} For any $\mathcal{F}\in \mathcal{H}_{\Lambda}(X)$ there is a neighborhood $U$ of $\mathcal{F}$ such that $ J_{\mathcal{F}}\subset J_{\mathcal{G}}$ for $\mathcal{F}\in U$.\\Take $$C_{V}=\{\mathcal{F}\in \mathcal{H}_{\Lambda}(X):  J_{\mathcal{F}}= J_{\mathcal{G}} for~\mathcal{G}~ sufficiently~ close~ to~ \mathcal{F}\}. $$ By definition $C_{V}$ is an open subset of $\mathcal{H}_{\Lambda}(X)$. Now we show that $C_{V}$ is dense in
$\mathcal{H}_{\Lambda}(X)$. Consider an arbitrary open set $W\subset \mathcal{H}_{\Lambda}(X)$ and $J_{\mathcal{F}}$ is a maximal element of $J_{W}=\{J_{\mathcal{F}}:\mathcal{F}\in W\}$, i.e., for every $\mathcal{G}\in W$, $J_{\mathcal{F}}\subset J_{\mathcal{G}}$ implies that $ \mathcal{F}=\mathcal{G}$. Thus, by claim $\mathcal{F}\in C_{V}\cap W$. So $C_{V}$ is an open dense subset of $\mathcal{H}_{\Lambda}(X)$.\\ Take $\mathcal{F}\in C_{V}$ we prove that $\mathcal{F}$ has the weak shadowing property. \\ Since $C_{V}$ is an open set there is $\gamma>0$ such that for $\mathcal{G}\in \mathcal{H}_{\Lambda}(X)$, $\rho( \mathcal{F},\mathcal{G})< \gamma$ implies that $J_{\mathcal{F}}= J_{\mathcal{G}}$. Suppose $\beta>0$ is a $\gamma-$modulus of homogeneity of $X$. Let $\mathrm{x}=\{x_{i}\}_{i\in \mathbb{Z}}$ be a $\frac{\beta}{2}-$pseudo orbit of $\mathcal{F}$. Because of the Lemma \ref{le2} there exist $k\in\mathbb{N}$ such that $\mathrm{x}\subset B_{\epsilon}(\mathrm{x}_{k})$.
By Lemma \ref{le1} there exist a $\beta-$pseudo trajectory $\mathbf{y}=\{y_{i}\}_{i=-n}^{n}$ ( belong to $\mathcal{F}$) such that $\mathbf{x_{k}}\subset B_{\epsilon}(\mathbf{y})$ and $y_{i}\neq y_{j}$ for all $i\neq j$. Suppose $0<\tau<\gamma$ is a number that $d(a,b)<\tau$ implies that $d(f^{-1}_{\lambda}(a),f^{-1}_{\lambda}(b))<\gamma $, for all $a,b\in X$ and all $\lambda\in \Lambda$. Also, suppose $h\in\mathcal{H}(X)$, $d_{0}(h,id_{X})<\tau$ is a homeomorphism connecting $f_{\lambda_{i}}(y_{i})$ with $y_{i+1}$ for all $-n\leq i\leq n-1$. Consider IFS $\mathcal{G}=\{X; g_{\lambda}=h o f_{\lambda}|\lambda\in\Lambda\}$  and let $\sigma= \{\mu_{0},\mu_{1},...\}$ be an arbitrary element of  $\Lambda^{\mathbb{Z}_{+}}$. So the sequence $$\mathbf{z}=\{...,g^{-1}_{\mu_{1}}(g^{-1}_{\mu_{0}}(y_{-n})),g^{-1}_{\mu_{0}}(y_{-n}),y_{-n},y_{-(n-1)},...,y_{n},g_{\mu_{0}}(y_{n}),g_{\mu_{1}}(g_{\mu_{0}}(y_{n})),...\}$$
is an orbit of $\mathcal{G}$.
\\ This is clear that $\rho( \mathcal{F},\mathcal{G})< \gamma$ and hence $J_{\mathcal{F}}= J_{\mathcal{G}}$. So there ia an orbit $\mathbf{z}^{'}$  of $\mathcal{F}$ such that for any $1\leq i\leq k$ if $\mathbf{z}^{'}\cap V_{i}\neq \emptyset$ then  $\mathbf{z}\cap V_{i}\neq \emptyset$. Thus $\mathbf{z}\subset B_{\epsilon}(\mathbf{z}^{'})$ and consequently $\mathbf{z}^{'}\subset B_{3\epsilon}(\mathbf{x})$.
\end{proof}
\textbf{Proof of Claim}: Suppose that $J_{F}=\{L_{1},L_{2},...,L_{m}\}$. For any $1\leq j\leq m$, there is an orbit $\mathrm{x}^{j}$ such that $\mathrm{x}^{j}\cap U_{j_{i}}\neq \emptyset$ for all $j_{i}\in L_{j}$. So there exist $\epsilon_{j}>0$ such that $\rho( \mathcal{F},\mathcal{G})<\epsilon_{j}$ implies that, for an orbit $\mathrm{y}^{j}$ of $ \mathcal{G}$ such that $\mathrm{y}^{j}\cap U_{j_{i}}\neq \emptyset$ for all $j_{i}\in L_{j}$. Thus $L_{j}\in $. Take $\epsilon=\min\{\epsilon_{1},\epsilon_{2},...,\epsilon_{m}\}$, similar argument shows that if $\rho( \mathcal{F},\mathcal{G})<\epsilon$ then $L_{1}\in J_{G}$ for all $1\leq j\leq m$.\\
By Theorem \ref{th1} and proof of Theorem 2 in \cite{[MM]} we have the following theorem.
\begin{theorem}  If the space $X$ is one of the following:\\
(i) a topological manifold with boundary ($dim(X)\geq 2$ if $\partial X\neq\emptyset$.),
\\
(ii) a Cartesian product of a countably infinite number of manifolds with nonempty
boundary,
\\
(iii) a Cantor set,\\
then $\mathcal{F}$-weak shadowing is a generic property in $H_{\Lambda}(X)$.
\end{theorem}

\end{document}